\documentclass{amsart}
\usepackage{amsmath}
\usepackage{amsfonts}

\newtheorem{theorem}{Theorem}
\theoremstyle{plain}

\newtheorem{corollary}{Corollary}

\newtheorem{definition}{Definition}

\newtheorem{lemma}{Lemma}

\newtheorem{remark}{Remark}

\numberwithin{equation}{section}
\input{tcilatex}

\begin{document}
\title{INEQUALITIES ON GEOMETRICALLY CONVEX FUNCTIONS}
\author{M.EM\.{I}N \"{O}ZDEM\.{I}R}
\address{ATAT\"{U}RK UNIVERSITY, K. K. EDUCATION FACULTY, DEPARTMENT OF
MATHEMATICS, 25240, CAMPUS, ERZURUM, TURKEY}
\email{emos@atauni.edu.tr}

\begin{abstract}
In this paper, we obtain some new upper bounds for differantiable mappings
whose q-th powers are geometrically convex and monotonically decreasing by
using the H\"{o}lder inequality, Power mean inequality and properties of
modulus. 
\end{abstract}

\maketitle

\section{INTRODUCTION}

The following double inequality is well known in the literature as
Hadamard's inequality:

Let $f:I\subseteq 
\mathbb{R}
\rightarrow 
\mathbb{R}
$ be a convex function defined on an interval $I$ of real numbers, $a,b\in I$
and $a<b,$ we have%
\begin{equation}
f\left( \frac{a+b}{2}\right) \leq \frac{1}{b-a}\dint\limits_{a}^{b}f(x)dx%
\leq \frac{f(a)+f(b)}{2}.  \label{222}
\end{equation}%
Both inequalities hold in the reversed direction if $f$ is concave.

It was first discovered by Hermite in 1881 in the Journal Mathesis (see \cite%
{herm}). The inequality (\ref{222}) was nowhere mentioned in the
mathematical literature until 1893. Beckenbach, a leading expert on the
theory of convex functions, wrote that inequality (\ref{222}) was proven by
Hadamard in 1893 (see \cite{beck}). In 1974 Mitrinovi\v{c} found
HermiteHermite's note in Mathesis. That is why, the inequality (\ref{222})
was known as Hermite-Hadamard inequality.

A function $f:[a,b]\subset 
\mathbb{R}
\rightarrow 
\mathbb{R}
$ is said to be convex if whenever $x,y\in \lbrack a,b]$ and $t\in \lbrack
0,1]$, the following inequality holds:%
\begin{equation*}
f(tx+(1-t)y)\leq tf(x)+(1-t)f(y).
\end{equation*}%
We say that $f$ is concave if ($-f$) is convex. This definition has its
origins in Jensen's results from \cite{d} and has opened up the most
extended, useful and multi-disciplinary domain of mathematics, namely,
convex analysis. Convex curves and convex bodies have appeared in
mathematical literature since antiquity and there are many important results
related to them.

In \cite{b}, the concept of geometrically convex functions were introduced
as following:

\begin{definition}
A function $f:I\subset 
\mathbb{R}
_{+}\rightarrow 
\mathbb{R}
_{+}$ is said to be a geometrically convex function if%
\begin{equation*}
f\left( x^{t}y^{1-t}\right) \leq \left[ f(x)\right] ^{t}\left[ f(y)\right]
^{1-t}
\end{equation*}%
for all $x,y\in I$ and $t\in \lbrack 0,1]$.
\end{definition}

For some recent results connected with geometrically convex functions, see 
\cite{b1}-\cite{b}.

\begin{definition}
Let $a,b\in 
\mathbb{R}
,$ $a,b\neq 0$ and $\left\vert a\right\vert \neq \left\vert b\right\vert .$
Logarithmic mean for real numbers was inroduced as follows:%
\begin{equation*}
L(a,b)=\frac{a-b}{\ln \left\vert a\right\vert -\ln \left\vert b\right\vert }.
\end{equation*}
\end{definition}

In \cite{zdemir}, \"{O}zdemir and Y\i ld\i z established the following
Theorem:

\begin{theorem}
Let $f:I^{\circ }\subset 
\mathbb{R}
_{+}\rightarrow 
\mathbb{R}
_{+}$ be differentiable function on $I^{\circ }$, $a,b\in I$ with $a<b$ and $%
f^{\prime }\in L_{1}[a,b].$ If $\left\vert f^{\prime }\right\vert ^{q}$ is
geometrically convex and monotonically decreasing on $[a,b]$ and $t\in
\lbrack 0,1],$ then we have the following inequality:%
\begin{eqnarray}
&&\left\vert f(x)-\frac{1}{b-a}\int_{a}^{b}f(u)du\right\vert  \label{9} \\
&\leq &\frac{1}{(p+1)^{\frac{1}{p}}}\left\{ \frac{(x-a)^{2}}{b-a}\left[
L\left( \left\vert f^{\prime }(x)\right\vert ^{q},\left\vert f^{\prime
}(a)\right\vert ^{q}\right) \right] ^{\frac{1}{q}}+\frac{(b-x)^{2}}{b-a}%
\left[ L\left( \left\vert f^{\prime }(x)\right\vert ^{q},\left\vert
f^{\prime }(b)\right\vert ^{q}\right) \right] ^{\frac{1}{q}}\right\}  \notag
\end{eqnarray}%
where $\frac{1}{p}+\frac{1}{q}=1.$ $L($ $,$ $)$ is Logarithmic mean for real
numbers.
\end{theorem}

In order to prove our main results we need the following lemma (see \cite{a}%
).

\begin{lemma}
\label{1}Let $f:I\subset 
\mathbb{R}
\rightarrow 
\mathbb{R}
$ be a differentiable mapping on $I^{\circ }$ where $a,b\in I$ with $a<b$.
If $f^{\prime \prime }\in L_{1}[a,b],$ then the following equality holds:%
\begin{eqnarray*}
&&\frac{1}{b-a}\int_{a}^{b}f(u)du-f(x)+\left( x-\frac{a+b}{2}\right)
f^{\prime }(x) \\
&=&\frac{(x-a)^{3}}{2(b-a)}\int_{0}^{1}t^{2}f^{\prime \prime }(tx+(1-t)a)dt+%
\frac{(b-x)^{3}}{2(b-a)}\int_{0}^{1}t^{2}f^{\prime \prime }(tx+(1-t)b)dt
\end{eqnarray*}%
for each $x\in \lbrack a,b].$
\end{lemma}

In \cite{HME}, in order to prove some inequalities related to
Hermite-Hadamard inequality, Kavurmac\i\ \textit{et al}.\ used the following
lemma.

\begin{lemma}
\label{7}Let $f:I\subset 
\mathbb{R}
\rightarrow 
\mathbb{R}
$ be a differentiable mapping on $I^{o}$ where $a,b\in I$ with $a<b.$ If $%
f^{\prime }\in L[a,b],$ then the following equality holds:%
\begin{eqnarray*}
&&\frac{(b-x)f(b)+(x-a)f(a)}{b-a}-\frac{1}{b-a}\overset{b}{\underset{a}{\int 
}}f(u)du \\
&=&\frac{(x-a)^{2}}{b-a}\int_{0}^{1}(1-t)f^{\prime }(tx+(1-t)a)dt+\frac{%
(b-x)^{2}}{b-a}\int_{0}^{1}(t-1)f^{\prime }(tx+(1-t)b)dt.
\end{eqnarray*}
\end{lemma}

The main aim of this paper is to establish new inequalities for
geometrically convex functions.

\section{MAIN RESULTS}

\begin{theorem}
\label{cet}Let $f:I^{\circ }\subset 
\mathbb{R}
_{+}\rightarrow 
\mathbb{R}
_{+}$ be differentiable function on $I^{\circ }$, $a,b\in I$ with $a<b$ and $%
f^{\prime \prime }\in L_{1}[a,b].$ If $\left\vert f^{\prime \prime
}\right\vert ^{q}$ is geometrically convex and monotonically decreasing on $%
[a,b]$ and $t\in \lbrack 0,1],$ then we have the following inequality:%
\begin{eqnarray}
&&  \label{2} \\
&&\left\vert \frac{1}{b-a}\int_{a}^{b}f(u)du-f(x)+\left( x-\frac{a+b}{2}%
\right) f^{\prime }(x)\right\vert  \notag \\
&\leq &\frac{1}{(2p+1)^{\frac{1}{p}}}\left\{ \frac{(x-a)^{3}}{2(b-a)}\left[
L\left( \left\vert f^{\prime \prime }(x)\right\vert ^{q},\left\vert
f^{\prime \prime }(a)\right\vert ^{q}\right) \right] ^{\frac{1}{q}}+\frac{%
(b-x)^{3}}{2(b-a)}\left[ L\left( \left\vert f^{\prime \prime }(x)\right\vert
^{q},\left\vert f^{\prime \prime }(b)\right\vert ^{q}\right) \right] ^{\frac{%
1}{q}}\right\}  \notag
\end{eqnarray}%
where $1<p<\infty ,$ $\frac{1}{p}+\frac{1}{q}=1.$ $L($ $,$ $)$ is
Logarithmic mean for real numbers.
\end{theorem}

\begin{proof}
From Lemma \ref{1} with properties of modulus \ and using the H\"{o}lder
inequality, we have%
\begin{eqnarray}
&&  \label{3} \\
&&\left\vert \frac{1}{b-a}\int_{a}^{b}f(u)du-f(x)+\left( x-\frac{a+b}{2}%
\right) f^{\prime }(x)\right\vert  \notag \\
&\leq &\frac{(x-a)^{3}}{2(b-a)}\int_{0}^{1}t^{2}\left\vert f^{\prime \prime
}(tx+(1-t)a)\right\vert dt+\frac{(b-x)^{3}}{2(b-a)}\int_{0}^{1}t^{2}\left%
\vert f^{\prime \prime }(tx+(1-t)b)\right\vert dt  \notag
\end{eqnarray}%
\begin{eqnarray}
&\leq &\frac{(x-a)^{3}}{2(b-a)}\left( \int_{0}^{1}t^{2}{}^{p}dt\right) ^{%
\frac{1}{p}}\left( \int_{0}^{1}\left\vert f^{\prime \prime
}(tx+(1-t)a)\right\vert ^{q}dt\right) ^{\frac{1}{q}}  \notag \\
&&+\frac{(b-x)^{3}}{2(b-a)}\left( \int_{0}^{1}t^{2}{}^{p}dt\right) ^{\frac{1%
}{p}}\left( \int_{0}^{1}\left\vert f^{\prime \prime }(tx+(1-t)b)\right\vert
^{q}dt\right) ^{\frac{1}{q}}  \notag \\
&=&\frac{1}{(2p+1)^{\frac{1}{p}}}\left\{ \frac{(x-a)^{3}}{2(b-a)}\left(
\int_{0}^{1}\left\vert f^{\prime \prime }(tx+(1-t)a)\right\vert
^{q}dt\right) ^{\frac{1}{q}}\right.  \notag \\
&&\left. +\frac{(b-x)^{3}}{2(b-a)}\left( \int_{0}^{1}\left\vert f^{\prime
\prime }(tx+(1-t)b)\right\vert ^{q}dt\right) ^{\frac{1}{q}}\right\} .  \notag
\end{eqnarray}

Since $\left\vert f^{\prime \prime }\right\vert ^{q}$ is geometrically
convex and monotonically decreasing on $[a,b],$ we obtain 
\begin{eqnarray*}
x^{t}a^{1-t} &\leq &tx+(1-t)a \\
\left\vert f^{\prime \prime }(tx+(1-t)a)\right\vert ^{q} &\leq &\left\vert
f^{\prime \prime }(x^{t}a^{1-t})\right\vert ^{q}.
\end{eqnarray*}%
Therefore, we have 
\begin{eqnarray}
I &=&\int_{0}^{1}\left\vert f^{\prime \prime }(tx+(1-t)a)\right\vert ^{q}dt
\label{4} \\
&\leq &\int_{0}^{1}\left\vert f^{\prime \prime }(x^{t}a^{1-t})\right\vert
^{q}dt  \notag \\
&\leq &\int_{0}^{1}\left[ \left\vert f^{\prime \prime }(x)\right\vert
^{t}\left\vert f^{\prime \prime }(a)\right\vert ^{1-t}\right] ^{q}dt  \notag
\\
&=&L\left( \left\vert f^{\prime \prime }(x)\right\vert ^{q},\left\vert
f^{\prime \prime }(a)\right\vert ^{q}\right)  \notag
\end{eqnarray}%
and%
\begin{equation}
\int_{0}^{1}t^{2}{}^{p}dt=\frac{1}{2p+1}.  \label{10}
\end{equation}%
By making use of inequalities (\ref{10}) and (\ref{4}) in (\ref{3}), we
obtain (\ref{2}). This completes the proof.
\end{proof}

\begin{corollary}
Since $\frac{1}{3}<\frac{1}{(2p+1)^{\frac{1}{p}}}<1,$ if we choose $%
\left\vert f^{\prime \prime }(a)\right\vert =\left\vert f^{\prime \prime
}(b)\right\vert $ in Theorem \ref{cet}, we have%
\begin{eqnarray*}
&&\left\vert \frac{1}{b-a}\int_{a}^{b}f(u)du-f(x)+\left( x-\frac{a+b}{2}%
\right) f^{\prime }(x)\right\vert \\
&\leq &\frac{(x-a)^{3}+(b-x)^{3}}{2(b-a)}\left[ L\left( \left\vert f^{\prime
\prime }(x)\right\vert ^{q},\left\vert f^{\prime \prime }(a)\right\vert
^{q}\right) \right] ^{\frac{1}{q}}.
\end{eqnarray*}
\end{corollary}

\begin{theorem}
\label{tin}Let $f:I\subset 
\mathbb{R}
\rightarrow 
\mathbb{R}
_{+}$ be differentiable function on $I^{\circ }$, $a,b\in I$ with $a<b$ and $%
f^{\prime \prime }\in L[a,b].$ If $\left\vert f^{\prime \prime }\right\vert
^{q}$ is geometrically convex and monotonically decreasing on $[a,b]$ for $%
q\geq 1$ and $t\in \lbrack 0,1],$ then the following inequality holds:%
\begin{eqnarray}
&&\left\vert \frac{1}{b-a}\int_{a}^{b}f(u)du-f(x)+\left( x-\frac{a+b}{2}%
\right) f^{\prime }(x)\right\vert   \label{5} \\
&\leq &\left( \frac{1}{3}\right) ^{1-\frac{1}{q}}\left\{ \frac{(x-a)^{3}}{%
2(b-a)}\left( \frac{k}{\ln k}-\frac{2k}{\left( \ln k\right) ^{2}}+\frac{2k}{%
\left( \ln k\right) ^{3}}\right) ^{\frac{1}{q}}\right.   \notag \\
&&\text{ \ \ \ \ \ \ \ \ \ \ \ \ }\left. +\frac{(b-x)^{3}}{2(b-a)}\left( 
\frac{l}{\ln l}-\frac{2l}{\left( \ln l\right) ^{2}}+\frac{2l}{\left( \ln
l\right) ^{3}}\right) ^{\frac{1}{q}}\right\}   \notag
\end{eqnarray}%
where%
\begin{equation*}
k=\frac{\left\vert f^{\prime \prime }(x)\right\vert ^{q}}{\left\vert
f^{\prime \prime }(a)\right\vert ^{q}}\text{ \ and \ \ }l=\frac{\left\vert
f^{\prime \prime }(x)\right\vert ^{q}}{\left\vert f^{\prime \prime
}(b)\right\vert ^{q}}.
\end{equation*}
\end{theorem}

\begin{proof}
From Lemma \ref{1}\ and using the well-known power-mean inequality, we have%
\begin{eqnarray*}
&&\left\vert \frac{1}{b-a}\int_{a}^{b}f(u)du-f(x)+\left( x-\frac{a+b}{2}%
\right) f^{\prime }(x)\right\vert  \\
&\leq &\frac{(x-a)^{3}}{2(b-a)}\int_{0}^{1}t^{2}\left\vert f^{\prime \prime
}(tx+(1-t)a)\right\vert dt+\frac{(b-x)^{3}}{2(b-a)}\int_{0}^{1}t^{2}\left%
\vert f^{\prime \prime }(tx+(1-t)b)\right\vert dt \\
&\leq &\frac{(x-a)^{3}}{2(b-a)}\left( \int_{0}^{1}t^{2}dt\right) ^{1-\frac{1%
}{q}}\left( \int_{0}^{1}t^{2}\left\vert f^{\prime \prime
}(tx+(1-t)a)\right\vert ^{q}dt\right) ^{\frac{1}{q}} \\
&&+\frac{(b-x)^{3}}{2(b-a)}\left( \int_{0}^{1}t^{2}dt\right) ^{1-\frac{1}{q}%
}\left( \int_{0}^{1}t^{2}\left\vert f^{\prime \prime }(tx+(1-t)b)\right\vert
^{q}dt\right) ^{\frac{1}{q}} \\
&=&\left( \frac{1}{3}\right) ^{1-\frac{1}{q}}\left\{ \frac{(x-a)^{3}}{2(b-a)}%
\left( \int_{0}^{1}t^{2}\left\vert f^{\prime \prime }(tx+(1-t)a)\right\vert
^{q}dt\right) ^{\frac{1}{q}}\right.  \\
&&\left. +\frac{(b-x)^{3}}{2(b-a)}\left( \int_{0}^{1}t^{2}\left\vert
f^{\prime \prime }(tx+(1-t)b)\right\vert ^{q}dt\right) ^{\frac{1}{q}%
}\right\} .
\end{eqnarray*}%
Since $\left\vert f^{\prime \prime }\right\vert ^{q}$ is geometrically
convex and monotonically decreasing on $[a,b],$ we have%
\begin{eqnarray*}
&&\left\vert \frac{1}{b-a}\int_{a}^{b}f(u)du-f(x)+\left( x-\frac{a+b}{2}%
\right) f^{\prime }(x)\right\vert  \\
&\leq &\left( \frac{1}{3}\right) ^{1-\frac{1}{q}}\left\{ \frac{(x-a)^{3}}{%
2(b-a)}\left( \int_{0}^{1}t^{2}\left[ \left\vert f^{\prime \prime
}(x)\right\vert ^{t}\left\vert f^{\prime \prime }(a)\right\vert ^{1-t}\right]
^{q}dt\right) ^{\frac{1}{q}}\right.  \\
&&\left. +\frac{(b-x)^{3}}{2(b-a)}\left( \int_{0}^{1}t^{2}\left[ \left\vert
f^{\prime \prime }(x)\right\vert ^{t}\left\vert f^{\prime \prime
}(b)\right\vert ^{1-t}\right] ^{q}dt\right) ^{\frac{1}{q}}\right\} .
\end{eqnarray*}%
By integration by parts, we have the inequality (\ref{5}).
\end{proof}

\begin{corollary}
From Theorem \ref{cet} and Theorem \ref{tin}, we have%
\begin{equation*}
\left\vert \frac{1}{b-a}\int_{a}^{b}f(u)du-f(x)+\left( x-\frac{a+b}{2}%
\right) f^{\prime }(x)\right\vert \leq \min \left\{ v_{1},v_{2}\right\} 
\end{equation*}%
where%
\begin{equation*}
v_{1}=\frac{1}{(2p+1)^{\frac{1}{p}}}\left\{ \frac{(x-a)^{3}}{2(b-a)}\left[
L\left( \left\vert f^{\prime \prime }(x)\right\vert ^{q},\left\vert
f^{\prime \prime }(a)\right\vert ^{q}\right) \right] ^{\frac{1}{q}}+\frac{%
(b-x)^{3}}{2(b-a)}\left[ L\left( \left\vert f^{\prime \prime }(x)\right\vert
^{q},\left\vert f^{\prime \prime }(b)\right\vert ^{q}\right) \right] ^{\frac{%
1}{q}}\right\} 
\end{equation*}%
and%
\begin{equation*}
v_{2}=\left( \frac{1}{3}\right) ^{1-\frac{1}{q}}\left\{ \frac{(x-a)^{3}}{%
2(b-a)}\left( \frac{k}{\ln k}-\frac{2k}{\left( \ln k\right) ^{2}}+\frac{2k}{%
\left( \ln k\right) ^{3}}\right) ^{\frac{1}{q}}+\frac{(b-x)^{3}}{2(b-a)}%
\left( \frac{l}{\ln l}-\frac{2l}{\left( \ln l\right) ^{2}}+\frac{2l}{\left(
\ln l\right) ^{3}}\right) ^{\frac{1}{q}}\right\} 
\end{equation*}%
\begin{equation*}
k=\frac{\left\vert f^{\prime \prime }(x)\right\vert ^{q}}{\left\vert
f^{\prime \prime }(a)\right\vert ^{q}}\text{ \ and \ \ }l=\frac{\left\vert
f^{\prime \prime }(x)\right\vert ^{q}}{\left\vert f^{\prime \prime
}(b)\right\vert ^{q}}.
\end{equation*}
\end{corollary}

\begin{theorem}
\label{cett}Let $f:I^{\circ }\subset 
\mathbb{R}
_{+}\rightarrow 
\mathbb{R}
_{+}$ be differentiable function on $I^{\circ }$, $a,b\in I$ with $a<b$ and $%
f^{\prime }\in L_{1}[a,b].$ If $\left\vert f^{\prime }\right\vert ^{q}$ is
geometrically convex and monotonically decreasing on $[a,b]$ and $t\in
\lbrack 0,1],$ then we have the following inequality:%
\begin{eqnarray}
&&  \label{6} \\
&&\left\vert \frac{(b-x)f(b)+(x-a)f(a)}{b-a}-\frac{1}{b-a}%
\int_{a}^{b}f(u)du\right\vert   \notag \\
&\leq &\frac{1}{(p+1)^{\frac{1}{p}}}\left\{ \frac{(x-a)^{2}}{b-a}\left[
L\left( \left\vert f^{\prime }(x)\right\vert ^{q},\left\vert f^{\prime
}(a)\right\vert ^{q}\right) \right] ^{\frac{1}{q}}+\frac{(b-x)^{2}}{b-a}%
\left[ L\left( \left\vert f^{\prime }(x)\right\vert ^{q},\left\vert
f^{\prime }(b)\right\vert ^{q}\right) \right] ^{\frac{1}{q}}\right\}   \notag
\end{eqnarray}%
where $1<p<\infty ,$ $\frac{1}{p}+\frac{1}{q}=1.$ $L($ $,$ $)$ is
Logarithmic mean for real numbers.
\end{theorem}

\begin{proof}
From Lemma \ref{7} with properties of absolute value\ and using the H\"{o}%
lder inequality, we obtain%
\begin{eqnarray*}
&&\left\vert \frac{(b-x)f(b)+(x-a)f(a)}{b-a}-\frac{1}{b-a}%
\int_{a}^{b}f(u)du\right\vert  \\
&\leq &\frac{(x-a)^{2}}{b-a}\int_{0}^{1}\left\vert 1-t\right\vert \left\vert
f^{\prime }(tx+(1-t)a)\right\vert dt+\frac{(b-x)^{2}}{b-a}%
\int_{0}^{1}\left\vert t-1\right\vert \left\vert f^{\prime
}(tx+(1-t)b)\right\vert dt \\
&\leq &\frac{(x-a)^{2}}{b-a}\left( \int_{0}^{1}(1-t{})^{p}dt\right) ^{\frac{1%
}{p}}\left( \int_{0}^{1}\left\vert f^{\prime }(tx+(1-t)a)\right\vert
^{q}dt\right) ^{\frac{1}{q}} \\
&&+\frac{(b-x)^{2}}{b-a}\left( \int_{0}^{1}(1-t{})^{p}dt\right) ^{\frac{1}{p}%
}\left( \int_{0}^{1}\left\vert f^{\prime }(tx+(1-t)b)\right\vert
^{q}dt\right) ^{\frac{1}{q}} \\
&=&\frac{1}{(p+1)^{\frac{1}{p}}}\left\{ \frac{(x-a)^{2}}{b-a}\left(
\int_{0}^{1}\left\vert f^{\prime }(tx+(1-t)a)\right\vert ^{q}dt\right) ^{%
\frac{1}{q}}\right.  \\
&&\left. +\frac{(b-x)^{2}}{b-a}\left( \int_{0}^{1}\left\vert f^{\prime
}(tx+(1-t)b)\right\vert ^{q}dt\right) ^{\frac{1}{q}}\right\} .
\end{eqnarray*}

Since $\left\vert f^{\prime }\right\vert ^{q}$ is geometrically convex and
monotonically decreasing on $[a,b],$ we obtain 
\begin{eqnarray*}
x^{t}a^{1-t} &\leq &tx+(1-t)a \\
\left\vert f^{\prime \prime }(tx+(1-t)a)\right\vert ^{q} &\leq &\left\vert
f^{\prime \prime }(x^{t}a^{1-t})\right\vert ^{q}.
\end{eqnarray*}%
Therefore, we have%
\begin{eqnarray*}
K &=&\int_{0}^{1}\left\vert f^{\prime }(tx+(1-t)a)\right\vert ^{q}dt \\
&\leq &\int_{0}^{1}\left\vert f^{\prime }(x^{t}a^{1-t})\right\vert ^{q}dt \\
&\leq &\int_{0}^{1}\left[ \left\vert f^{\prime }(x)\right\vert
^{t}\left\vert f^{\prime }(a)\right\vert ^{1-t}\right] ^{q}dt \\
&=&L\left( \left\vert f^{\prime }(x)\right\vert ^{q},\left\vert f^{\prime
}(a)\right\vert ^{q}\right) 
\end{eqnarray*}%
and%
\begin{equation*}
\int_{0}^{1}(1-t{})^{p}dt=\frac{1}{p+1}.
\end{equation*}%
This completes the proof.
\end{proof}

\begin{remark}
In Theorem \ref{cett}, if we take $f(x)=\frac{(b-x)f(b)+(x-a)f(a)}{b-a},$ we
obtain the inequality (\ref{9}).
\end{remark}

\begin{corollary}
Since $\frac{1}{(p+1)^{\frac{1}{p}}}<\allowbreak 1$ for $1<p<\infty $, if we
choose $\left\vert f^{\prime }(a)\right\vert =\left\vert f^{\prime
}(b)\right\vert $ in Theorem \ref{cett}, we have%
\begin{eqnarray*}
&&\left\vert \frac{(b-x)f(b)+(x-a)f(a)}{b-a}-\frac{1}{b-a}%
\int_{a}^{b}f(u)du\right\vert  \\
&\leq &\frac{(x-a)^{2}+(b-x)^{2}}{b-a}\left[ L\left( \left\vert f^{\prime
}(x)\right\vert ^{q},\left\vert f^{\prime }(a)\right\vert ^{q}\right) \right]
^{\frac{1}{q}}.
\end{eqnarray*}
\end{corollary}

\begin{theorem}
\label{yil}Let $f:I\subset 
\mathbb{R}
\rightarrow 
\mathbb{R}
_{+}$ be differentiable function on $I^{\circ }$, $a,b\in I$ with $a<b$ and $%
f^{\prime }\in L[a,b].$ If $\left\vert f^{\prime }\right\vert ^{q}$ is
geometrically convex and monotonically decreasing on $[a,b]$ for $q\geq 1$
and $t\in \lbrack 0,1],$ then the following inequality holds:%
\begin{eqnarray}
&&\left\vert \frac{(b-x)f(b)+(x-a)f(a)}{b-a}-\frac{1}{b-a}%
\int_{a}^{b}f(u)du\right\vert   \label{8} \\
&\leq &\left( \frac{1}{2}\right) ^{1-\frac{1}{q}}\left\{ \frac{(x-a)^{2}}{b-a%
}\left\vert f^{\prime }(a)\right\vert \left( \frac{k-\log k-1}{\left( \log
k\right) ^{2}}\right) ^{\frac{1}{q}}+\frac{(b-x)^{2}}{b-a}\left\vert
f^{\prime }(b)\right\vert \left( \frac{l-\log l-1}{\left( \log l\right) ^{2}}%
\right) ^{\frac{1}{q}}\right\}   \notag
\end{eqnarray}%
where%
\begin{equation*}
k=\frac{\left\vert f^{\prime }(x)\right\vert ^{q}}{\left\vert f^{\prime
}(a)\right\vert ^{q}}\text{ \ and \ \ }l=\frac{\left\vert f^{\prime
}(x)\right\vert ^{q}}{\left\vert f^{\prime }(b)\right\vert ^{q}}.
\end{equation*}
\end{theorem}

\begin{proof}
From Lemma \ref{7}\ and using the well-known power-mean inequality, we have%
\begin{eqnarray*}
&&\left\vert \frac{(b-x)f(b)+(x-a)f(a)}{b-a}-\frac{1}{b-a}%
\int_{a}^{b}f(u)du\right\vert  \\
&\leq &\frac{(x-a)^{2}}{b-a}\int_{0}^{1}\left\vert 1-t\right\vert \left\vert
f^{\prime }(tx+(1-t)a)\right\vert dt+\frac{(b-x)^{2}}{b-a}%
\int_{0}^{1}\left\vert 1-t\right\vert \left\vert f^{\prime
}(tx+(1-t)b)\right\vert dt \\
&\leq &\frac{(x-a)^{2}}{b-a}\left( \int_{0}^{1}(1-t)dt\right) ^{1-\frac{1}{q}%
}\left( \int_{0}^{1}(1-t)\left\vert f^{\prime }(tx+(1-t)a)\right\vert
^{q}dt\right) ^{\frac{1}{q}} \\
&&+\frac{(b-x)^{2}}{b-a}\left( \int_{0}^{1}(1-t)dt\right) ^{1-\frac{1}{q}%
}\left( \int_{0}^{1}(1-t)\left\vert f^{\prime }(tx+(1-t)b)\right\vert
^{q}dt\right) ^{\frac{1}{q}}.
\end{eqnarray*}%
Since $\left\vert f^{\prime }\right\vert ^{q}$ is geometrically convex and
monotonically decreasing on $[a,b],$ we have%
\begin{eqnarray*}
&&\left\vert \frac{(b-x)f(b)+(x-a)f(a)}{b-a}-\frac{1}{b-a}%
\int_{a}^{b}f(u)du\right\vert  \\
&\leq &\left( \frac{1}{2}\right) ^{1-\frac{1}{q}}\left\{ \frac{(x-a)^{2}}{b-a%
}\left( \int_{0}^{1}(1-t)\left\vert f^{\prime }(x^{t}a^{1-t})\right\vert
^{q}dt\right) ^{\frac{1}{q}}\right.  \\
&&\text{ \ \ \ \ \ }\left. +\frac{(b-x)^{2}}{b-a}\left(
\int_{0}^{1}(1-t)\left\vert f^{\prime }(x^{t}b^{1-t})\right\vert
^{q}dt\right) ^{\frac{1}{q}}\right\}  \\
&\leq &\left( \frac{1}{2}\right) ^{1-\frac{1}{q}}\left\{ \frac{(x-a)^{2}}{b-a%
}\left( \int_{0}^{1}(1-t)\left( \left\vert f^{\prime }(x)\right\vert
^{t}\left\vert f^{\prime }(a)\right\vert ^{1-t}\right) ^{q}dt\right) ^{\frac{%
1}{q}}\right.  \\
&&\text{ \ \ \ \ }\left. +\frac{(b-x)^{2}}{b-a}\left(
\int_{0}^{1}(1-t)\left( \left\vert f^{\prime }(x)\right\vert ^{t}\left\vert
f^{\prime }(b)\right\vert ^{1-t}\right) ^{q}dt\right) ^{\frac{1}{q}}\right\}
.
\end{eqnarray*}%
This completes the proof.
\end{proof}

\begin{corollary}
Since $\frac{1}{2}<\left( \frac{1}{2}\right) ^{1-\frac{1}{q}}<1$, if we
choose $\left\vert f^{\prime }(a)\right\vert =\left\vert f^{\prime
}(b)\right\vert $ in Theorem \ref{yil}, we have%
\begin{eqnarray*}
&&\left\vert \frac{(b-x)f(b)+(x-a)f(a)}{b-a}-\frac{1}{b-a}%
\int_{a}^{b}f(u)du\right\vert  \\
&\leq &\left( \frac{k-\log k-1}{\left( \log k\right) ^{2}}\right) ^{\frac{1}{%
q}}\left\vert f^{\prime }(a)\right\vert +\frac{(x-a)^{2}+(b-x)^{2}}{b-a}.
\end{eqnarray*}
\end{corollary}

\begin{corollary}
From Theorem \ref{cett} and Theorem \ref{yil}, again we have%
\begin{equation*}
\left\vert \frac{1}{b-a}\int_{a}^{b}f(u)du-f(x)+\left( x-\frac{a+b}{2}%
\right) f^{\prime }(x)\right\vert \leq \min \left\{ \eta _{1},\eta
_{2}\right\} 
\end{equation*}%
where%
\begin{equation*}
\eta _{1}=\frac{1}{(p+1)^{\frac{1}{p}}}\left\{ \frac{(x-a)^{2}}{b-a}\left[
L\left( \left\vert f^{\prime }(x)\right\vert ^{q},\left\vert f^{\prime
}(a)\right\vert ^{q}\right) \right] ^{\frac{1}{q}}+\frac{(b-x)^{2}}{b-a}%
\left[ L\left( \left\vert f^{\prime }(x)\right\vert ^{q},\left\vert
f^{\prime }(b)\right\vert ^{q}\right) \right] ^{\frac{1}{q}}\right\} 
\end{equation*}%
and%
\begin{equation*}
\eta _{2}=\left( \frac{1}{2}\right) ^{1-\frac{1}{q}}\left\{ \frac{(x-a)^{2}}{%
b-a}\left\vert f^{\prime }(a)\right\vert \left( \frac{k-\log k-1}{\left(
\log k\right) ^{2}}\right) ^{\frac{1}{q}}+\frac{(b-x)^{2}}{b-a}\left\vert
f^{\prime }(b)\right\vert \left( \frac{l-\log l-1}{\left( \log l\right) ^{2}}%
\right) ^{\frac{1}{q}}\right\} 
\end{equation*}%
\begin{equation*}
k=\frac{\left\vert f^{\prime }(x)\right\vert ^{q}}{\left\vert f^{\prime
}(a)\right\vert ^{q}}\text{ \ and \ \ }l=\frac{\left\vert f^{\prime
}(x)\right\vert ^{q}}{\left\vert f^{\prime }(b)\right\vert ^{q}}.
\end{equation*}
\end{corollary}

\end{document}